\documentclass{amsart}

\usepackage{geometry}
\usepackage{enumerate}
\usepackage{amsmath}
\usepackage{mathrsfs}
\usepackage{color}
\usepackage{amssymb}

\usepackage[normalem]{ulem}

\usepackage[colorlinks=true,urlcolor=blue, citecolor=blue,linkcolor=blue,linktocpage,pdfpagelabels,bookmarksnumbered,bookmarksopen]{hyperref}

\numberwithin{equation}{section}

\newtheorem{theorem}{Theorem}[section]
\newtheorem{lemma}[theorem]{Lemma}
\newtheorem{corollary}[theorem]{Corollary}
\newtheorem{remark}[theorem]{Remark}
\newtheorem{definition}[theorem]{Definition}

\newtheorem{proposition}[theorem]{Proposition}
\newtheorem{example}[theorem]{Example}

\begin{document}
\title[Zeta-function residues]{Extended zeta-function residues on principal ideals}

\author[Tian]{Yongqiang Tian}
\address{School of Mathematics and Statistics, Central South University, Changsha 410075, China}
\email{tianyongqiang@csu.edu.cn}

\author[Usachev]{Alexandr Usachev}
\address{School of Mathematics and Statistics, Central South University, Changsha 410075, China}
\email{dr.alex.usachev@gmail.com}



\thanks{Y. T. was supported by the Fundamental Research Funds for the Central University of Central South University (No.2022ZZTS0605). A. U. was supported by the Theoretical Physics and Mathematics Advancement Foundation ``BASIS''}

\begin{abstract}
We study extended zeta-function residues on principal ideals of compact operators and their connections with Dixmier traces. We establish a Lidskii-type formula for continuous singular traces on these ideals. Using this formula, we obtain a necessary and sufficient conditions for an arbitrary operator being Dixmier measurable. These conditions are expressed in terms of eigenvalues of an operator and an asymptotic of its zeta-function. 
\end{abstract}

\maketitle

\section{Introduction}
Let $\mathcal{L}_\infty(H)$ (respectively, $\mathcal{C}_0(H)$) denote the algebra of all bounded (respectively, compact) linear operators  on a complex separable Hilbert
space $H$. By $\mu(T):=\{\mu(n,T)\}_{n=0}^{\infty}$ we denote the sequence of singular values of an operator $T\in\mathcal{C}_0(H).$
Define the weak trace class ideal $\mathcal{L}_{1,\infty}$ and the Dixmier-Macaev ideal $\mathcal{M}_{1,\infty}$, respectively, by setting
\begin{equation*}
  \mathcal{L}_{1,\infty}:=\Big\{ T\in \mathcal{C}_0(H):~\mu(n,T)=O\left((n+1)^{-1}\right)\Big\}
\end{equation*}
and
\begin{equation*}
  \mathcal{M}_{1,\infty}:=\left\{ T\in \mathcal{C}_0(H):~\sum_{k=0}^n\mu(k,T)=O\big(\log(n+2)\big)\right\}.
\end{equation*}
In 1966, J. Dixmier constructed the first example of singular (that is non-normal) trace \cite{Dixmier1966}. For a certain state $\omega$ on the algebra $\ell_\infty$ of all bounded sequences, he proved that the weight
\begin{equation*}
 \mathrm{Tr}_{\omega}(T):=\omega\left(n\mapsto\frac{1}{\log(n+2)}\sum_{k=0}^n\mu(k,T)\right),\quad 0\leq T\in \mathcal{M}_{1,\infty}
\end{equation*}
extends by linearity to a trace on the whole ideal $\mathcal{M}_{1,\infty}$ (in fact, he proved this in more general case of Lorentz ideals (see Definition \ref{lorentz_def} below)). Nowadays theses traces are called Dixmier traces. Being an analogue of classical integration, Dixmier traces play the key role in Connes' noncommutative calculus and geometry, see \cite{Connes1994}, \cite{Connes1988},\cite{LSZ2013} and \cite{G01}.

The computation of Dixmier traces can be an involved task. One reason for this is the dependence of Dixmier traces on the generating state $\omega$. The classical result of A. Connes states that all Dixmier traces coincide on classical pseudo-differential operators of certain order \cite{Connes1988}. More recently, it turned out that this is not generally true without the classicality condition (see, e.g. \cite[Corollary 2.5.7]{LMSZ_book}). Now, it is also known that there are large classes of operators on which Dixmier traces do not coincide (e.g. H\"ormander-Weyl pseudo-differential operators \cite[Corollary 5.7]{GS2014}, Hankel operators on Hardy spaces \cite{GU2020_2}, commutators with non-smooth functions \cite{Gimperlein_Goffeng, GU2021}, etc.) The other (and more substantial) reason is that for a generic operator appearing from geometrical or physical problems one generally does not have an access to its singular values. There were several techniques developed to overcome this obstacle. They aimed to produce formulae which express Dixmier (and more generally, singular) traces in terms of:

- eigenvalues of an operator (Lidskii-type formulae);

- asymptotic of operator zeta-function;

- asymptotic of partition function;

- expectation values.

In this paper we deal with the first two techniques. We discuss them in more details.

Let $\mathcal{L}_{1}$ be the set of trace class operators on the Hilbert space $H$. In 1959, V. Lidskii established a formula expressing the classical trace in terms of the eigenvalues of the operator \cite{Lidskii1959} (see also \cite{Power1983}):
\begin{equation*}
	\mathrm{Tr}(T)=\sum_{j=0}^\infty\lambda(j,T),\quad  T\in\mathcal{L}_1.
\end{equation*}
Formulating traces in terms of eigenvalues has been shown achievable on more general ideals. For an ideal $\mathcal{I}$ in $\mathcal{L}_\infty(H)$ and $T\in\mathcal{I}$, regarding $\lambda(T)$ as a diagonal operator with respect to some fixed orthonormal basis in $H,$ a remarkable result in \cite{spectral} states that: the equality $\varphi(T)=\varphi(\lambda(T)),~T\in\mathcal{I}$ holds for all traces $\varphi$ on $\mathcal{I}$ if and only if $\mathcal{I}$ is \emph{logarithmically closed} (see Definition \ref{log}).
Although all ideals which appear in this paper are logarithmically closed, a trace, say Dixmier trace, is usually realized by a concrete formula involving singular values. So it is still worth rewriting the trace formula by using eigenvalues. There were several attempts recorded in the literature (see, e.g. \cite{AS2005, SSZ2010, SUZ2014}).

Some more recent advances in the theory of singular traces were based on a dyadic approach to traces discovered by A. Pietsch \cite{Pietsch1} for ideals in $\mathcal L_\infty$ and later extended to the setting of semifinite von Neumann algebras \cite{LU2023}. A modification of this approach established a bijection between all traces on $\mathcal L_{1,\infty}$ and all shift-invariant functionals on the space $\ell_\infty$ of all bounded sequences \cite{SSUZ2015}. This result allowed to resolve some open problems in the theory. Using this bijection the following Lidskii-type formula was established in \cite{SSUZ2015}:
\begin{equation}\label{dyadic}
	\varphi(T)=\theta\left(\frac{1}{\log 2}\left\{\sum_{j=2^n-1}^{2^{n+1}-2} \lambda(j, T)\right\}_{n \geq 0}\right),\quad T\in\mathcal{L}_{1,\infty},
\end{equation}
where $\varphi$ is an arbitrary trace on $\mathcal{L}_{1,\infty}$ and $\theta$ is the shift-invariant functional corresponding to it. In Section 3 we prove the first main result of this paper (Theorem \ref{Lid} below). It extends formula \eqref{dyadic} to cover more general eigenvalue asymptotics. Besides its general interest, we use this analogue of Lidskii's formula in Section 5 to describe operators on which various classes of singular traces coincide.

In Sections 4 and 5 we study the relation between Dixmier traces and the operator zeta-functions. For a positive operator $T\in \mathcal{M}_{1,\infty}$ the function
\begin{equation}\label{zzeta}
	t\mapsto \mathrm{Tr}(T^{1+1/t})
\end{equation}
is well-defined for $t>0$ and is called the operator zeta-function.
 Although it is not evident that there is any connection between the behaviour of zeta-function and Dixmier traces, the following surprising result was established by A. Connes and H. Moscovici \cite[Proposition~A.4]{CM1995}.
\begin{theorem}\rm{}\label{equi}
 Let $0\leq T\in\mathcal{L}_{1,\infty},$ the equality $\mathrm{Tr}_\omega(T)=c$ holds for all Dixmier traces if and only if
 $$ \lim_{t\to +\infty}\frac{1}{t}\mathrm{Tr}(T^{1+1/t}) = c.$$
 \end{theorem}
 That is, the common value of Dixmier traces on a operator $T$ can be read from the asymptotic of its zeta-function at the leading singularity.
 This fundamental result has been extended and generalised by various authors. A priori,  the limit in Theorem \ref{equi} need not to exist. In this case formulae relating Dixmier traces and asymptotic of zeta-function were established for positive operators in ideal $\mathcal{M}_{1,\infty}$ and a ``small'' subclass of Dixmier traces in \cite{CPS2003} and in \cite{CRSS2007} for a larger subclass of Dixmier traces. To cover the case of non-positive operators it is natural to consider operators of the form $VT$ with an arbitrary bounded $V$ and a positive $T$ from the ideal. This was achieved for operators in the ideal $\mathcal M_{1,\infty}$ and a subclass of Dixmier traces (generated by exponentiation invariant $\omega$'s, to be precise) in \cite{SZ2011}. It was further extended to more general Lorentz ideals in \cite{GS2014}.

 The modified dyadic approach mention above, enabled the extension of the Connes-Moscovici result for the class of \textit{all} Dixmier traces and operators of the form $VT$ with $V\in\mathcal{L}_\infty,0\leq T\in\mathcal{L}_{1,\infty}$ \cite{SUZ2017}.

More recently, G. Levitina and the second author extended the modified dyadic approach to more general  ideals \cite{principal}. For a decreasing function $g:[0,\infty)\rightarrow(0,\infty)$ satisfying
\begin{equation}\label{cond1}
\lim_{t\rightarrow\infty}\frac{g(t)}{g(2t)}=\frac12,
\end{equation}
they considered the following principal ideal
$$\mathcal{I}_g:=\{T\in\mathcal{C}_0(H):\mu(n,T)=O(g(n)) \}.$$
They established a bijective correspondence between continuous singular traces on $\mathcal{I}_g$ and continuous shift-invariant functional on $\ell_\infty$. Such correspondence allows to reformulate positive normalised traces  in terms of well-known and fairly well-studied Banach limits.

In Sections 4 and 5 we deal with extensions of Theorem \ref{equi} to more general ideals.
The initial question in this direction is how to define extended zeta-function residue on $\mathcal{I}_g$.
Motivated by \cite{CRSS2007} and \cite{GS2014}, we consider the following function
\begin{equation}\label{zzzeta}
t\mapsto\frac{1}{G(e^t)}\mathrm{Tr}(T^{1+1/t}),\quad 0\leq T\in \mathcal{I}_g,
\end{equation}
where $G$ is a primitive of $g$.
Although this function is bounded for $t>0$, its convergence as $t\to\infty$ is not guaranteed. However, after an application of some extended limit $\gamma$ on $L_\infty$ and making some appropriate restrictions on $g$, the resulting functional extends by linearity from the positive cone to a singular trace on the whole ideal. This extended functional is called an extended zeta-function residue on $\mathcal{I}_g$. We show that this functional is a scalar multiple of a Dixmier trace. Thus, in view of the bijection between continuous singular traces on $\mathcal{I}_g$ and continuous shift-invariant functional on $\ell_\infty$, this functional correspond to some Banach limit. Under several assumptions on the function $g$ we establish the form of this Banach limit (Lemma \ref{blmt0} and Theorem \ref{blmt}). Blending these results with the Lidskii formula from Section 3, we prove a criterion for coincidence of all extended zeta-function residues on an arbitrary operator from principal ideal $\mathcal{I}_g$.

In Section 5 we limit our considerations to principal ideals $\mathcal{I}_k,~k\geq0$ generated by an operator $T\in \mathcal L_\infty$ such that $$\mu(n,T)\sim\frac{\log^kn}{n},\quad n\to\infty.$$
These ideals accommodate several interesting classes of operators. For instance, operators from scattering calculus (of order $(-n, -n)$ on $\mathbb R^n$) belongs to $\mathcal I_1$ \cite[Example 2.7]{GU2020}, \cite{Nicola2003}, the (inverse of) Dirichlet Laplacian for a certain region in $\mathbb R^n$, $n\ge2$ of infinite volume belongs to $\mathcal I_{n-1}$ \cite{Camus2015} and log-polyhomogeneous operators \cite{Lesch_log_PDO} are found in all $\mathcal{I}_k,~k\geq0$.
In this section we extend Theorem \ref{equi} to the case of non-positive operators from $\mathcal{I}_k,~k\geq0$. Using the Lidskii formula, proved in Section 3, we are able to state the result in terms of eigenvalues. This is the second main result of this paper (Theorem \ref{main} below):
\begin{theorem}\label{main_Intro}
	Let $V\in\mathcal{L}_{\infty},~0\leq T\in\mathcal{I}_k$ and $c\in\mathbb{C}.$ The following are equivalent:
	\begin{enumerate}[(i)]
		\item All Dixmier traces equal $c$ on $VT$, that is
		$$\lim_{n\rightarrow\infty}\frac{1}{ G_{k}(n+1)}\sum_{j=0}^n\lambda(j,VT)=c.$$
		\item All extended $\zeta$-function residues coincide on operator $VT.$
		\item
		\begin{equation*}
			\lim_{t\rightarrow\infty}\frac{1}{t^{k+1}}\mathrm{Tr}(VT^{1+1/t})=k!c.
		\end{equation*}
	\end{enumerate}
\end{theorem}

The organization of the paper is the following. In Section 2 we gather some necessary terminology and auxiliary results. In Section 3 we employ a dyadic approach to traces to prove the Lidskii-type formula for all continuous singular traces on the principal ideals $\mathcal{I}_g$ (Theorem \ref{Lid}). In Section 4 we establish a particular form of the shift-invariant functional on $\ell_\infty$ corresponding to extended $\zeta$-function residues on $\mathcal{I}_g$ (Theorem \ref{blmt}). This allows to prove a criterion of measurability with respect to extended zeta-function residues on $\mathcal{I}_g$ in terms of the power series summation (generalization of the classical Abel summation). In Section 5 we prove Theorem \ref{main_Intro}.
As an application, we derive a sufficient condition for a tensor product (consisting of operators in $\mathcal{I}_0=\mathcal{L}_{1,\infty}$) being  Dixmier-measurable. We also study two examples of Dirichlet Laplacians and prove that they both belong to $\mathcal I_k$ and Dixmier-measurable.

\section{Preliminaries}
\subsection{Notations}
Set $\mathbb{N}:=\{ 0,1,2,\cdots\}$. We denote by $\ell_{\infty}$ the space of of all real-valued bounded sequences on $\mathbb{N}$ equipped with the norm
$$ \|x\|_{\infty}:=
\sup_{n\geq0}|x_n|.$$
We denote by $L_{\infty}$ the space of all (equivalence classes of) real-valued
essentially bounded Lebesgue measurable functions on $(0, \infty)$ equipped with the norm
$$ \|f\|_{\infty}:=\underset{{t>0}}{\mathrm{ess~sup}}~|f(t)|.$$
For $1\leq p<\infty$, we denote by $L_p$ the usual Lebesgue spaces on $(0,\infty)$.

For two measurable functions $f$ and $g$ defined on some neighborhood of $t_0$, the relation $f(t)\sim g(t),~t\rightarrow t_0$ means that
$$f(t)=g(t)+o(g(t)),\quad t\to t_0.$$
Recall that by $\mathcal{L}_\infty(H)$ ($\mathcal{C}_0(H),$ respectively) we denote the algebra of all bounded (compact, respectively) linear operators  on a separable Hilbert
space $H$. For brevity, we also write $\mathcal{L}_{\infty}:=\mathcal{L}_{\infty}(H)$. For an operator $T\in\mathcal{L}_\infty$, we denote by $\|T\|_\infty$ its operator norm. And we denote by $\Re T=\frac{1}{2}(T+T^*)$ and $\Im T=\frac{1}{2i}(T-T^*)$ its real and imaginary parts, respectively. For a self-adjoint operator $T\in\mathcal{L}_\infty$, we denote by $T_+$ and $T_-$ its positive and negative parts, and denote by $|T|=(T^*T)^{\frac12}$ the modulus of $T$.

Throughout the paper, we fix an orthonormal basis $\{e_n\}_{n\geq0}$ in $H.$ Using this basis, we define the diagonal homomorphism $\mathrm{diag}:\ell_{\infty}\rightarrow\mathcal{L}_{\infty}$  by setting
$$\mathrm{diag}(x):=\sum_{i=0}^{\infty}x_{n}e_{n}\otimes e_{n},~x\in \ell_\infty,$$
where, $e_{n}\otimes e_n$ is the rank one projection onto $\mathbb{C}e_n$.

For an operator $T\in\mathcal{C}_0(H),$ we denote by $\lambda(T):=\{\lambda(n,T)\}_{n\geq0}$ the sequence of eigenvalues (repeated according to multiplicity) of $T$ arranged such that the sequence of absolute values is non-increasing. The sequence $\mu(n,T):=\lambda(n,|T|),$ $n\geq0$ is referred to as the singular value sequence of the operator $T.$  Sometimes,
we also consider the singular value function $\mu(T):=\mu(\cdot,T)$ of $T$ defined by setting
$$\mu(t,T):=\sum_{n=0}^\infty\mu(n,T)\chi_{[n,n+1)}(t),\quad t\geq0.$$
Let $\mathrm{Tr}$ denote the standard trace on $\mathcal{L}_\infty.$ For $1\leq p<\infty,$ we denote by $\mathcal{L}_p$ the Schatten ideals of compact operators equipped with the norm $$\|T\|_p:=(\mathrm{Tr}(|T|^p))^{\frac1p}=\big(\sum_{n=0}^\infty(\mu(n,T)^p\big)^{\frac1p}.$$ 

\subsection{Functions of regular variation}
A concept of regularly varying functions play a crucial role in this paper, we list some relevant definitions and key properties.
We refer the reader to the book \cite[Section 1.4.2]{regular} for further details.
\begin{definition}\label{definition of rv function}
 A positive measurable function $g$ defined on some interval $(a,\infty)$ is said to be \emph{regular varying of index $\alpha$}($\in\mathbb{R}$) if
$$ \lim _{t \rightarrow \infty} \frac{g(\lambda t)}{g(t)}=\lambda^\alpha, \quad\forall~ \lambda>0.$$
The classes of \emph{regularly varying} fucntions is denoted by $\mathcal{R}_\alpha$. The class $\mathcal{R}_0$ is usually called that of \emph{slowly varying} functions.
\end{definition}

\begin{proposition}\label{testing}(\cite[Proposition~2.15]{GU2020})
Let $\alpha\in\mathbb{R}.$
For a monotone function $g:(a,\infty)\rightarrow(0,\infty),$ we have  $g\in \mathcal{R}_\alpha$ if and only if
$$\lim _{t \rightarrow \infty} \frac{g\left(\lambda_0 t\right)}{g(t)}=\lambda_0^\alpha,\quad ~for~some~\lambda_0>0.$$
\end{proposition}
In particular, if $g$ is a positive decreasing function satisfying \ref{cond1}, then $g\in\mathcal{R}_{-1}.$
\begin{definition}\label{definition of function class G}
Let $\mathcal{G}$ be the set of all decreasing and vanishing at $+\infty$ functions $g:[0,\infty)\rightarrow(0,\infty)$ such that $g\notin L_1$ and condition \eqref{cond1} holds.
Let $\Omega$ be the set of all increasing, concave, diverging at infinity functions  $G:[0,\infty)\rightarrow[0,\infty)$ such that $G(t)=O(t),$ $t\downarrow0$ and
\begin{equation}\label{cond02}
\lim_{t\rightarrow\infty}\frac{G(2t)}{G(t)}=1.
\end{equation}
\end{definition}
Note that functions in $\mathcal{G}$ or $\Omega$ are always monotone, taking $\lambda_0=2$ in Proposition \ref{testing}, we conclude that $\mathcal{G}\subset\mathcal{R}_{-1}$ and $\Omega\subset\mathcal{R}_0.$
It is straightforward that for $g\in\mathcal{G}$ its primitive $G(t):=\int_0^t g(s)ds,$ $t\geq0$ belongs to $\Omega.$
\subsection{Ideals and traces}
In this paper, we consider several partial orders on $\mathcal{L}_\infty$ \cite{AH1994}.
\begin{definition}\label{log}
Let $S,T\in\mathcal{L}_{\infty},$ then
\begin{enumerate}[(i)]
  \item  the operator $S$ is said to be \emph{submajorized} by $T$ (in the sense of Hardy-Littlewood-$P\acute{o}lya$) if and only if
$$\sum_{j=0}^{n}\mu(j,S)\leq\sum_{j=0}^{n}\mu(j,T),\quad n\geq0.$$
In this case, we write $S\prec\prec T$.
  \item  the operator $S$ is said to be \emph{logarithmically~submajorized} by $T$  if and only if
$$\prod_{j=0}^{n}\mu(j,S)\leq\prod_{j=0}^{n}\mu(j,T),\quad n\geq0.$$
In this case, we write $S\prec\prec_{\log} T$.
\end{enumerate}
\end{definition}
For every $T\in\mathcal{C}_0(H),$ we have the following Weyl inequality (see, e.g. \cite{WEYL1949})
\begin{equation}\label{weyl}
\prod_{k=0}^n|\lambda(k,T)|\leq\prod_{k=0}^n\mu(k,T),\quad n\geq0.
\end{equation}
In other words, $\mathrm{diag}(\lambda(T))\prec\prec_{\log} T.$
\begin{definition}\label{partial orders in ideals}
We say an ideal $\mathcal{I}$ in $\mathcal{L}_\infty$ is \begin{enumerate}[(i)]
\item \emph{fully symmetric}, if for $S\in\mathcal{L}_\infty,~T\in\mathcal{I}$ and $S\prec\prec T$, we have $S\in\mathcal{I}$.
\item \emph{logarithmically closed}, if for $S\in\mathcal{L}_\infty,~T\in\mathcal{I}$ and $S\prec\prec_{\log} T$, we have $S\in\mathcal{I}$.
\item \emph{geometrically stable}, if $$\mathrm{diag}\Big(\Big\{\Big(\prod_{j=0}^n\mu(j,T)\Big)^{\frac{1}{n+1}}\Big\}_{n=0}^{\infty}\Big)\in\mathcal{I},\quad \forall ~T\in\mathcal{I}.$$
\end{enumerate}
\end{definition}

\begin{remark}
It is well-known that for $p\geq1,$
\begin{equation}\label{fully symmetric}
\mathrm{Tr}(S^p)\leq \mathrm{Tr}(T^p),\quad S,T\geq0, \mathrm{~and~}S\prec\prec T.
\end{equation}
This implies that $\mathcal{L}_p,$ $p\geq1$ are fully symmetric ideals. Besides, the class of logarithmically ideals is strictly wider than that of gemetrically stable ideals, see \cite{spectral}.
\end{remark}

For every $g\in\mathcal{G},$ consider the principal ideal
$$\mathcal{I}_g:=\{T\in\mathcal{C}_0(H):\mu(n,T)=O(g(n))\}.$$
Condition \eqref{cond1} guarantees that the mapping
\begin{equation}\label{quasinorm}
\|T\|_{\mathcal{I}_g}:=\sup_{n\geq0}\frac{\mu(n,T)}{g(n)}
\end{equation}
is a quasi-norm (\cite[Theorem~1.2]{sparr}).

\begin{remark}\label{rmk1}
In definition of $\mathcal{I}_g$ we supposed that $g$ is decreasing on $[0,\infty).$ This assumption can be relaxed by that
$g$ is eventually decreasing, i.e., for some $a>0,$ the function
$$\tilde{g}:t\mapsto g(t+a),\quad t\geq0$$
is decreasing. In this case, we have $\mathcal{I}_g=\mathcal{I}_{\tilde{g}}$ with equivalent quasi-norms. The particular choice of $a>0$ is inessential.
\end{remark}

\begin{example}
For $k\in\mathbb{N}$, consider the principal ideals $\mathcal{I}_k:=\mathcal{I}_{g_k}$ with $g_k(t)=\frac{\log^{k}(t+2)}{t+2}.$
When $k=0$, it is easy to see that $\mathcal{I}_{0}$ coincides with the weak trace class ideal $\mathcal{L}_{1,\infty}.$
\end{example}

For every $G\in\Omega$, the Lorentz ideal $\mathcal{M}_G$ is defined by setting
\begin{equation}\label{lorentz_def}
	\mathcal{M}_G:=\left\{T\in \mathcal{C}_0(H):\|T\|_{\mathcal{M}_G}:=\sup_{n\geq0}\frac{1}{G(n+1)}\sum_{k=0}^n\mu(k,T)<\infty\right\}.
\end{equation}
Note that, if $G(t)=\log(t+1)$, then $\mathcal{M}_G$ is exactly the Dixmier-Macaev ideal $\mathcal{M}_{1,\infty}$ from Introduction. And, if $G\in\Omega$ is the primitive function for some $g\in\mathcal{G},$ then it is clear that $\mathcal{I}_g\subset\mathcal{M}_G.$ Moreover,
\begin{equation}\label{second definition for lorentz ideal}
\mathcal{M}_G=\left\{T\in \mathcal{C}_0(H):\mu(T)\prec\prec c\cdot g \mathrm{~for~some~c>0}\right\}.
\end{equation}

\begin{definition}
Let $\mathcal{I}$ be an ideal in $\mathcal{L}_{\infty}$. A linear functional $\varphi:\mathcal{I}\rightarrow\mathbb{C}$ is called a \emph{trace} if
$$ \varphi(ST)=\varphi(TS),\quad T\in\mathcal{I},~S\in\mathcal{L}_{\infty}.$$
This is equivalent to say that $\varphi$ is \emph{unitarily invariant}, i.e., $\varphi(U^*TU)=\varphi(T)$ holds whenever $T\in\mathcal{I}$ and $U$ is unitary (\cite[Lemma 1.2.1]{LSZ2013}).
\end{definition}
By Calkin theorem \cite{Calkin}, for any non-trivial ideal $\mathcal{I}$ in $\mathcal{L}_{\infty}$, the following inclusions hold
$$ \mathcal{F}(H)\subset\mathcal{I}\subset \mathcal{C}_0(H),$$
where $\mathcal{F}(H)$ is the set of finite rank operators on $H.$
\begin{definition}
Let $\mathcal{I}$ be an ideal in $\mathcal{L}_\infty,$ a linear functional $\varphi:\mathcal{I}\rightarrow\mathbb{C}$ is called
\begin{enumerate}[(i)]
\item \emph{positive} if $\varphi(T)\geq0,$ for all $0\leq T\in\mathcal{I}.$
\item \emph{singular} if $\varphi$ vanishes on $\mathcal{F}(H).$
\end{enumerate}
\end{definition}

Below, we recall an equivalent description of positive traces (see, e.g. \cite[Lemma 14]{spectral}).
\begin{lemma}\label{positive trace}
Let $\mathcal{I}$ be an ideal in $\mathcal{L}_\infty.$ A linear functional $\varphi$ on $\mathcal{I}$ is a positive trace if and only if $\varphi(S)\leq \varphi(T)$ for all $0\leq S,T\in\mathcal{I}$ such that $\mu(S)\leq\mu(T).$
\end{lemma}

The following lemma shows, that how the asymptotic behavior (at infinity) of singular values for a positive operator effect its trace.
\begin{lemma}
Let $\mathcal{I}\nsubseteq\mathcal{L}_1$ be an ideal in $\mathcal{L}_{\infty}$, and let $\varphi$ be a positive trace on $\mathcal{I}.$ If $0\leq S,T\in \mathcal{I}$ are such that $\mu(n,S)\sim\mu(n,T)$ as $n\rightarrow\infty$, then $\varphi(S)=\varphi(T).$
\end{lemma}
\begin{proof}\rm{}
By assumption, for every $\varepsilon>0,$ there exists $N>0$ such that $$\mu(S)\leq(1+\varepsilon)\mu(T)+\|S\|_\infty \chi_{[0,N]}.$$
Since $\varphi$ is positive and $\mathcal{I\nsubseteq}\mathcal{L}_1,$ it follows from \cite[Lemma 15(a)]{spectral} that $\varphi$ is singular. This together with Lemma \ref{positive trace} implies that
$$\varphi(S)\leq (1+\varepsilon)\varphi(T).$$
Since $\varepsilon$ is arbitrary, we have $\varphi(S)\leq\varphi(T).$ Interchanging $S$ and $T,$ we have $\varphi(S)\geq\varphi(T).$ It follows that $\varphi(S)=\varphi(T).$
\end{proof}

\begin{definition}
A positive trace $\varphi$ on some Lorentz ideal $\mathcal{M}_G$ is said to be \emph{fully symmetric} if $\varphi(S)\leq\varphi(T)$ whenever $0\leq S,T\in \mathcal{M}_{G}$ and $S\prec\prec T.$
\end{definition}
\subsection{Extended limits}
In this subsection, we collection some necessary definitions and results about extended limits on sequences and functions.
\begin{definition}
The left and right shift operators $S_l, S_r:\ell_\infty\rightarrow\ell_\infty$ are defined respectively by
$$S_{l}(x_0,x_1,x_2,\cdots)=(x_1,x_2,x_3\cdots),$$
$$S_{r}(x_0,x_{1},x_{2},\cdots)=(0,x_0,x_{1},x_{2},\cdots).$$
Clearly, we have $S_l\circ S_r=I.$
\end{definition}
\begin{definition}
 A linear functional $\theta$ on $\ell_\infty$ is said to be \emph{shift-invariant} if $\theta\circ S_r=\theta$.
\end{definition}
It is easy to see that a linear functional $\theta$ on $\ell_\infty$ is shift-invariant if and only if $\theta\circ S_l=\theta$ (see, e.g. \cite[Lemma~8.1]{Pietsch2015}).

Let us recall a notion of extended limit on $\ell_\infty$, i.e., the Hahn-Banach extension of the classical limit functional from the subspace of convergent sequences to the whole space $\ell_\infty$.
\begin{definition}
Let $\theta\in\ell_{\infty}^*,$ then $\theta$ is called an \emph{extended limit} on $\ell_\infty$ if
\begin{enumerate}[(i)]
\item $\theta(x)\geq0,$ provided that $x\geq0;$
\item $\theta(x)=\lim_{n\rightarrow\infty}x_n,$ provided $x$ is convergent.
\end{enumerate}
\end{definition}
It is well-known that
$$ \{\theta(x):\theta\in EL(\mathbb{N})\}=[\liminf_{n\rightarrow\infty}x_{n},\hspace{0.2em}\limsup_{n\rightarrow\infty}x_{n}].$$

Let $\sigma_2:\ell_\infty\rightarrow\ell_\infty$ be the dilation operator defined by
$$\sigma_2x=(x_0,x_0,x_1,x_1,x_2,x_2,\cdots).$$
An extended limit $\omega\in EL(\mathbb{N})$ is called \emph{dilation invariant} if $\omega\circ \sigma_2=\omega.$ There is an extensive theory of dilation invariant extended limits and their connection with Dixmier traces (see, e.g. \cite{SUZ2013, CPS2003, SSUZ2019, ASSU2016, ASSU2018}).

Below we introduce a notion of Banach limit, which will be frequently used in this paper. This definition appeared in works of S. Mazur and S. Banach \cite{Mazur, Banach1932}. For more details and recent advances we refer to \cite{SSU_Notices, Sofi_Survey, Das_Nanda_book, SSU2019, SSU2011}.
\begin{definition}\label{Blimit}
An extended limit $B:\ell_{\infty}\rightarrow\mathbb{R}$ is called a Banach limit if it is shift-invariant.
\end{definition}
Extended limits on $L_\infty$ are defined in a similar fashion.
\begin{definition}
A functional $\gamma\in L_\infty^*$ is said to be an \emph{extended limit} (at $+\infty$) if:
\begin{enumerate}[(i)]
 \item $\gamma(f)\geq0,$ provided that $f\geq0;$
\item $\gamma(f)=ess\lim_{t\rightarrow\infty}f(t),$ provided $f$ is convergent at $+\infty$.
\end{enumerate}
\end{definition}
The sets of all extended limits on $\ell_{\infty}$ and $L_{\infty}$ are denoted by $EL(\mathbb{N})$ and $EL(0,\infty)$, respectively.

The following lemma will be frequently used, see e.g. \cite[Lemma 9.3.6, 6.3.5, 6.3.1]{LSZ2013}.
\begin{lemma}\label{triple}
Let $\gamma\in EL(0,\infty)$. We have
\begin{enumerate}[(i)]
  \item\label{ela} for every $f\in L_\infty,$
  \begin{equation*}
  \{\gamma(f):\gamma\in EL(0,\infty)\}=[ess\liminf_{t\rightarrow\infty}f(t),\hspace{0.2em}ess\limsup_{t\rightarrow\infty}f(t)].
\end{equation*}
  \item\label{elb} for every $f\in L_\infty$ such that $\gamma(f)=a$ and $f\geq a,$
  \begin{equation*}
    \gamma(fg)=a\gamma(g),\quad \forall~g\in L_\infty.
  \end{equation*}
  \item\label{elc} if $\tilde{\gamma}$ is the additive extension of $\gamma$ on $ L_0^+ (0,\infty)$ the set of positive measurable functions (see \cite[Lemma 6.5.1]{LSZ2013}) given by
 $$\tilde{\gamma}(g):=\sup\{\gamma(f):0\leq f\leq g,~f\in L_\infty\},~g\in L_0^+(0,\infty).$$
 And if $g$ is bounded on $(t,\infty)$ for some $t>0$, then $\tilde{\gamma}(g)=\gamma(g\chi_{(t,\infty)}).$
\end{enumerate}
\end{lemma}
\begin{proof}
We only need to prove \eqref{elc}. By assumption, $g\chi_{(t,\infty)}\in L_\infty,$ it is clear that $\tilde{\gamma}(g)\geq \gamma(g\chi_{(t,\infty)}).$ On the other hand,
for every $0\leq f\leq g$ with $f\in L_\infty,$ we have
$$ f\leq f\chi_{(0,t]}+g\chi_{(t,\infty)}\leq g.$$ Note that $\gamma$ is an extended limit, we have $\gamma(f\chi_{(0,t]})=0.$ So $\gamma(f)\leq\gamma(g\chi_{(t,\infty)}).$ Since $f$ is arbitrary, we deduce that $ \tilde{\gamma}(g)\leq \gamma(g\chi_{(t,\infty)}).$ In conclusion,
$\tilde{\gamma}(g)=\gamma(g\chi_{(t,\infty)}).$
\end{proof}
In particular, $\tilde{\gamma}(t\mapsto2^{1+1/t})=2.$ For simplicity of notation, we shall write $\gamma(g):=\tilde{\gamma}(g)$ whenever $g\in L_0^+(0,\infty).$

\subsection{Dixmier traces and zeta-function residues}
\begin{definition}\label{dix}
Let $G\in\Omega.$ A functional $\mathrm{Tr}_\omega$ on $\mathcal{M}_G$ is called a \emph{Dixmier trace} if it is a linear extension of a weight
$$\mathrm{Tr}_{\omega}(T):=\omega\left( n\mapsto\frac{1}{ G (n+1) }\sum_{j=0}^{n}\mu(j,T)\right),~~0\leq T\in \mathcal{M}_G$$
for some dilation-invariant extended limit $\omega\in EL(\mathbb{N}).$
\end{definition}

Let $g\in\mathcal{G},$ and let $ G\in\Omega$ be its primitive function. Dixmier traces on the principal ideals $ \mathcal{I}_g$ are defined as restrictions of  Dimxier traces from $\mathcal{M}_G$. However, on principal ideals $\mathcal{I}_g,$ dilation invariance of $\omega$ is redundant. Indeed, by \cite[Theorem~17]{SS2013}, for every $\omega\in EL(\mathbb{N})$, the functional
$$\mathrm{Tr}_{\omega}(T):=\omega\left( n\mapsto\frac{1}{ G (n+1) }\sum_{j=0}^{n}\mu(j,T)\right),~~0\leq T\in \mathcal{I}_g$$
extends to a Dixmier trace on $\mathcal{M}_G$.
\subsection{Extended zeta-function residues}
\begin{proposition}\label{close to L_1}
Let $g\in\mathcal{G}.$ For every $p>1,$ the inclusion $\mathcal{I}_g\subset\mathcal{L}_p$ holds. Moreover, we have $\mathcal{I}_g\nsubseteq\mathcal{L}_1.$
\end{proposition}
\begin{proof}
The inclusion $\mathcal{I}_g\subset\mathcal{L}_p $ was established in \cite[Proposition~2.21]{GU2020}. The assertion $\mathcal{I}_g\nsubseteq\mathcal{L}_1$ follows from that $g\notin L_1$ since $g\in\mathcal{G}.$
\end{proof}
Let $G\in\Omega.$ If $G^{\prime}\in\mathcal{G},$ then the Lorentz ideal $\mathcal{M}_G$ is \emph{closed to $\mathcal{L}_1,$} that is, $\mathcal{M}_G$ is continuously embedded in $\mathcal{L}_p$ for all $p>1,$ and is not contained in $\mathcal{L}_1.$  Indeed, this assertion follows from a combination of Proposition \ref{close to L_1} and \cite[Lemma 2.2]{GS2014}.

Consider the following fully symmetric subspaces of $\mathcal{M}_G:$
$$\mathcal{L}_G:=\Big\{ T\in\mathcal{C}_0(H):\|T\|_{\mathcal{L}_G}:=\mathrm{ess}\sup_{t>0}\frac{\|T\|_{1+1/t}}{G(e^t)}<\infty\Big\}$$
and
$$\mathcal{Z}_G:=\Big\{ T\in\mathcal{C}_0(H):\|T\|_{\mathcal{Z}_G}:=\limsup_{t\to\infty}\frac{\mathrm{Tr}(|T|^{1+1/t})}{G(e^t)}<\infty\Big\}.$$
We infer that
$$\mathcal{L}_G\subset\mathcal{Z}_G\subset\mathcal{M}_G.$$
To see the inclusion $\mathcal{L}_G\subset\mathcal{Z}_G,$ we note that for a given operator $T\in\mathcal{L}_G,$ the following inequality $$\frac{\mathrm{Tr}(T^{1+1/t})}{G(e^t)}\leq G(e^t)^{1/t}\|T\|^{1+1/t}_{\mathcal{L}_G}$$
holds almost everywhere and is therefore, valid for all $t>0$ since both sides above are continuous functions of $t.$ By \cite[Lemma 4.1]{CRSS2007}, for every $\varepsilon>0,$ there is $C_{\varepsilon}>0$ such that $$G(t)<C_{\varepsilon}t^{\varepsilon},\quad t>0.$$
It follows that $G(e^t)^{1/t}\leq C_{\varepsilon}^{1/t}e^{\varepsilon},$ $t>0.$
Hence,
$$\limsup_{t\to\infty}\frac{\mathrm{Tr}(|T|^{1+1/t})}{G(e^t)}\leq  \lim_{t\to\infty}C_{\varepsilon}^{1/t}e^{\varepsilon}\|T\|_{\mathcal{L}_G}^{1+1/t}=e^{\varepsilon}\|T\|_{\mathcal{L}_G}<\infty.$$
The proof of inclusion $\mathcal{Z}_G\subset\mathcal{M}_G$ is similar to that of \cite[Theorem 4.5(i)]{CRSS2007}.

The extended zeta-function residues on ideals $\mathcal{L}_G$ were defined and studied by Gayral and Sukochev in \cite{GS2014} for a smaller class of $G,$ and the extended limits they used there possess certain invariance property. In the following, we extend this notion to $\mathcal{Z}_G$ in full generality.

\begin{definition}\label{zeta}
Let $G\in\Omega$ and let $\gamma\in EL(0,\infty).$ The following functional $\zeta_\gamma$ is said to be an \emph{extended zeta-function residue} on $\mathcal{Z}_G:$
\begin{equation*}
\zeta_\gamma(T):=\gamma\Big( t\mapsto\frac{1}{G(e^t)}\mathrm{Tr}\big(T^{1+1/t}\big)\Big),\quad 0\leq T\in\mathcal{Z}_G.
\end{equation*}
\end{definition}
\begin{lemma}
 $\zeta_\gamma$ is finite, positive homogeneous and additive on the positive cone of $\mathcal{Z}_G$.
\end{lemma}
\begin{proof}
The finiteness follows from the definition of $\mathcal{Z}_G.$
Now let $T\geq0$ and $\lambda>0.$ Since $\lambda^{1+1/t}\geq \lambda$ and $\gamma(t\mapsto\lambda^{1+1/t})=\lambda,$ it follow from Lemma \ref{triple} \eqref{elb},\eqref{elc} that $$\zeta_\gamma(\lambda T)=\gamma\left( t\mapsto\lambda^{1+1/t}\cdot\frac{1}{G(e^t)}\mathrm{Tr}\left(T^{1+1/t}\right)\right)=\lambda \zeta_\gamma(T).$$
This proves the homogeneity. Next, we recall the inequalities from \cite[Lemma~8.6.3]{LSZ2013} (the restriction $S,T\in\mathcal{M}_{1,\infty}$ is inessential there),$$ \mathrm{Tr}(S^{1+1/t}+T^{1+1/t})\leq \mathrm{Tr}((S+T)^{1+1/t})\leq 2^{1/t}\mathrm{Tr}(S^{1+1/t}+T^{1+1/t}),\quad t>0.$$
By the first inequality and the definition of $\zeta_\gamma,$ we have
$$\zeta_\gamma(S)+\zeta_\gamma(T)\leq \zeta_\gamma(S+T).$$
Since $\gamma(t\mapsto 2^{1/t})=1,$ by Lemma \ref{triple} \eqref{elb},\eqref{elc} and the second inequality, we have
$$\zeta_\gamma(S+T)\leq \gamma\left(t\mapsto2^{1/t}\cdot \frac{1}{G(e^t)}\mathrm{Tr}(S^{1+1/t}+T^{1+1/t})\right)=\zeta_{\gamma}(S)+\zeta_{\gamma}(T).$$
Finally, $$\zeta_{\gamma}(S+T)=\zeta_{\gamma}(S)+\zeta_{\gamma}(T).$$ This proves the additivity.
\end{proof}
The linear extension of $\zeta_\gamma$ on $\mathcal{Z}_G$ should be understood in the following sense
$$ \zeta_\gamma(T):=\zeta_\gamma(\Re(T)_+)-\zeta_\gamma(\Re(T)_-)+i\zeta_\gamma(\Im(T)_+)-i\zeta_\gamma(\Im(T)_-),\quad T\in\mathcal{Z}_G.$$
Let $G\in\Omega$. For every $V\in\mathcal{L}_{\infty},$ we define
\begin{equation}\label{vzeta}
 \zeta_{\gamma,V}(T):=\gamma\left(t\mapsto \frac {1}{G(e^t)}\mathrm{Tr}\left(VT^{1+1/t}\right)\right),\quad0\leq T\in\mathcal{Z}_G.
\end{equation}
Similarly to \cite[Theorem 8.6.5]{LSZ2013}, it can be shown that
\begin{equation}\label{Vzeta}
\zeta_{\gamma,V}(T)=\zeta_{\gamma}(VT),\quad V\in\mathcal{L}_{\infty},\quad 0\leq T\in\mathcal{Z}_G.
\end{equation}

\begin{proposition}\label{scalar}
Let $g\in\mathcal{G}$ and let $G\in\Omega$ be its primitive. If $\mathcal{I}_g\subset\mathcal{Z}_G,$ then every extended zeta-function residue on $\mathcal{I}_g$ is proportional to a Dixmier trace.
\end{proposition}
\begin{proof}
As is was mentioned above, $\mathcal{Z}_G$ is fully symmetric subspace of $\mathcal{M}_G.$ By \eqref{second definition for lorentz ideal}, the assumption $\mathcal{I}_g\subset\mathcal{Z}_G$ implies that $\mathcal{M}_G=\mathcal{Z}_G.$  By \eqref{fully symmetric}, every extended zeta-function residue is a fully symmetric functional on $\mathcal{M}_G.$  Using \cite[Theorem 11]{KSS2011}, we deduce that it is a scalar multiple of some Dixmier trace.
\end{proof}

Every normalised zeta-function residue is a Dixmier trace, but the converse is not true even on the well-studied ideal $\mathcal{L}_{1,\infty}$, see, e.g. \cite[Theorem~3.10]{SUZ2017}.

In this paper, we need a notion of measurability.
\begin{definition}
Let $g\in\mathcal{G}.$ An operator $T\in\mathcal{I}_g$ is said to be \emph{zeta-measurable} (respectively, \emph{Dixmier-measurable}) if all zeta-function residues (respectively, Dixmier traces) take the same value on $T.$
\end{definition}

The following theorem was proved in \cite[Theorem~20]{SS2013}.
\begin{theorem}\label{Dixmier measurable thm}
 Let $g\in\mathcal{G}$ and let $G\in \Omega$ be its primitive. Then an operator $T$ in $\mathcal{I}_g$ is Dixmier-measurable if and only if there exists the limit
$$\lim_{n\rightarrow\infty}\frac{1}{ G(n+1)}\sum_{j=0}^{n} \lambda(j,T).$$
\end{theorem}

\section{Lidskii formulae}
In this section, we prove an analogue of the classical Lidskii formula for continuous traces on $\mathcal{I}_g$. This result extends and generalises \cite[Theorem~4.1]{SSUZ2015} where it was proved for weak-trace class ideal $\mathcal L_{1,\infty}$.

Following \cite{principal}, for every $g\in\mathcal{G}$, we define the associated Piestch operator $D_g:\ell_\infty\rightarrow\mathcal{I}_g$ by setting
\begin{equation}\label{pop}
D_g(x_{0},x_{1},\cdots,x_n\cdots)=\mbox{diag} (x_{0}g(2^0),\underbrace{x_{1}g(2^1),x_{1}g(2^1)}_{2 ~\mathrm{times}},\cdots ,\underbrace{x_{n}g(2^n),\cdots,x_{n} g(2^n)}_{2^n~\mathrm{times}},\cdots).
\end{equation}

The following ono-to-one correspondence was established in \cite{principal} (see Theorem 3.5 there).
\begin{theorem}\label{biject}
Let $g\in\mathcal{G}$. There is a bijection between continuous singular traces on $\mathcal{I}_g$ and continuous shift-invariant functionals on $\ell_\infty$, given by the formulae
$$ \varphi(T)=\theta\Big (n\mapsto\frac{1}{2^n g(2^n)}\sum_{j=2^n-1}^{2^{n+1}-2}\mu(j,T) \Big),\quad 0\leq T\in\mathcal{I}_g,$$
$$\theta(x)=\varphi\circ D_{g} (x),~x\in\ell_\infty.$$
\end{theorem}
The following auxiliary result is useful in the sequel (see \cite[Lemma 2.19]{GU2020} ).
\begin{lemma}\label{shift lemma}
Let $g\in\mathcal{G}.$  If $f\in L_{\infty}(0,\infty)$ is such that
$$\int_{0}^{2^n}f(s)\mathrm{~d}s=O(2^ng(2^n)),\quad n\geq 0,$$
then the sequence
$$\left\{ \frac{1}{2^ng(2^n)}\int_{2^n}^{2^{n+1}}f(s)\mathrm{~d}s\right\}_{n\geq0}$$
belongs to $\mathrm{Range}(I-S_r)+c_0(\mathbb{N}),$ where $c_0(\mathbb{N})$ is the set of bounded sequences vanishing at infinity.
\end{lemma}

First, we prove the analogue of Lisdskii's formula for self-adjoint operators.
\begin{theorem}\label{self}
Let $g\in\mathcal{G}.$ If $T=T^{\ast}\in\mathcal{I}_g,$ and if $\varphi$ is a continuous singular trace on $\mathcal{I}_g,$ then
\begin{equation*}
  \varphi(T)=\theta\left(\left\{\frac{1}{2^{n}g(2^n)}\sum_{j=2^{n}-1}^{2^{n+1}-2}\lambda(j,T)\right\}_{n\geq0}\right),
\end{equation*}
where $\theta:=\varphi\circ D_g$ is the shift-invariant functional from Theorem \ref{biject}.
\end{theorem}

\begin{proof}\rm{}
Since $\varphi$ and $\theta$ are linear, it follows from Theorem \ref{biject}, that
\begin{equation}\label{formu}
  \varphi(T)=\varphi(T_+)-\varphi(T_-)=\theta\left(\left\{\frac{1}{2^{n}g(2^n)}\sum_{j=2^{n}-1}^{2^{n+1}-2}\mu(j,T_{+})-\mu(j,T_{-})\right\}_{n\geq0}\right),
\end{equation}
By \cite[Theorem 5.2.7]{LSZ2013}, we have
\begin{equation*}
  \left|\sum_{j=0}^{n}\lambda(j,T)-\mu(j,T_{+})-\mu(j,T_{-})\right|\leq 2(n+1)\mu(n,T)=O(ng(n)).
\end{equation*}
Note that $\theta$ is shift invariant and continuous.
Using Lemma \ref{shift lemma}, we obtain
$$ \theta\left(\left\{\frac{1}{2^n g(2^n)}\sum_{j=2^{n}-1}^{2^{n+1}-2}\lambda(j,T)-\mu(j,T_{+})-\mu(j,T_{-})\right\}_{n\geq0}\right)=0.$$
By linearity, we deduce that
$$\theta\left(\left\{\frac{1}{2^n g(2^n)}\sum_{j=2^{n}-1}^{2^{n+1}-2}\lambda(j,T)\right\}_{n\geq0}\right)=\theta\left(\left\{\frac{1}{2^n g(2^n)}\sum_{j=2^{n}-1}^{2^{n+1}-2}\mu(j,T_{+})-\mu(j,T_{-})\right\}_{n\geq0}\right).$$
Combining this with \eqref{formu} proves the assertion.
\end{proof}
The following analogue of Lidskii's formula is the first main result of the paper.
\begin{theorem}\rm{}\label{Lid}
 Let $g\in\mathcal{G}.$ If $\varphi$ is a continuous singular trace on $\mathcal{I}_g$ and $\theta:=\varphi\circ D_g$, then
\begin{equation}\label{Lidf}
  \varphi(T)= \theta\left(\left\{\frac{1}{2^n g(2^n)}\sum_{j=2^{n}-1}^{2^{n+1}-2}\lambda(j,T)\right\}_{n\geq0}\right),\quad~T\in\mathcal{I}_g.
\end{equation}
\end{theorem}
\begin{proof}\rm{}
Recall that for every $T\in\mathcal{C}_0(H),$ there is a well-known Ringrose decomposition: $T=N+Q$, where $Q$ is quasi-nilpotent and $N$ is normal with $\lambda(N)=\lambda(T)$ (see \cite{Ringrose1962}). Hence, $\mu(N)=|\lambda(T)|.$
This combined with the Weyl inequality \eqref{weyl} yields that
\begin{equation*}
  \prod_{j=0}^{n}\mu(j,N)=\prod_{j=0}^{n}|\lambda(j,T)|\leq\prod_{j=0}^{n}\mu(j,T),\hspace{0.3em},\quad n\geq0.
\end{equation*}
that is, $N\prec\prec_{\log} T.$ Clearly, $\mathcal{I}_g$ is a quasi-Banach ideal, so it is geometrically stable by \cite[Proposition 3.2]{Kalton1998},  and is therefore, logarithmically closed. It follows that $N\in\mathcal{I}_g$ and $Q=T-N\in\mathcal{I}_g.$ Since the spectrum of $Q$ consists of only zero, then by \cite[Theorem 8]{spectral}, we conclude that $\varphi(Q)=0.$
From Theorem \ref{self}, we deduce that
\begin{equation}\label{eq03}
  \varphi(T)=\varphi(N)=\varphi(\Re N)+i\varphi(\Im N)=\theta\left(\left\{\frac{1}{2^n g(2^n)}\sum_{j=2^{n}-1}^{2^{n+1}-2}\lambda(j,\Re N)+i\lambda(j,\Im N)\right\}_{n\geq0}\right).
\end{equation}
By \cite[Lemma 5.2.10 ]{LSZ2013}, we have
\begin{equation*}
  \left|\sum_{j=0}^{n}\lambda(j,N)-\lambda(j,\Re N)-i\lambda(j,\Im N)\right|\leq 5(n+1)\mu(n,N)=O(n g(n)).
\end{equation*}
By Lemma \ref{shift lemma}, we get
\begin{equation}\label{eq04}
\theta\left(\left\{\frac{1}{2^n g(2^n)}\sum_{j=2^{n}-1}^{2^{n+1}-2}\lambda(j,N)\right\}_{n\geq0}\right)=\theta\left(\left\{\frac{1}{2^n g(2^n)}\sum_{j=2^{n}-1}^{2^{n+1}-2}\lambda(j,\Re N)+i\lambda(j,\Im N)\right\}_{n\geq0}\right).
\end{equation}
Combining \eqref{eq03} and \eqref{eq04}, we conclude the proof.
\end{proof}

\section{Zeta-measurability}
In this section we additionally assume that
a function $G\in\Omega$ is such that $G\circ \exp\in \mathcal{R}_\alpha.$
This is equivalent to
\begin{equation}\label{Gexp}
  \lim_{t\rightarrow\infty}\frac{G(t^{\lambda})}{G(t)}=\lambda^\alpha,\quad \forall~\lambda>0.
\end{equation}
Hence, $G$ satisfies condition (8) from \cite{GS2014}.
The following should be compared to \cite[Theorem 3.3]{GS2014}. The proof is similar to \cite[Theorem 3.3]{GS2014} and therefore omitted.
\begin{lemma}\label{limit lemma}
Let $G$ satisfy \eqref{Gexp} with $\alpha\geq0,$ and let $0\leq T\in\mathcal{M}_G.$ If for some positive constant $c,$ there exists the limit
$$
\lim_{n\rightarrow\infty}\frac{1}{G(n+1)}\sum_{k=0}^n \mu(k,T)=c,
$$
then
$$
\lim_{t\rightarrow\infty}\frac{1}{G(e^t)}\mathrm{Tr}(T^{1+1/t})=c\cdot\Gamma(\alpha+1).
$$
Here, $\Gamma$ stands for the Gamma function.
\end{lemma}
This result enables us to compute the zeta-function residues of several special operators.
\begin{proposition}\label{special value proposition}
 Let $g\in\mathcal{G}$ an let $G\in\Omega$ be its primitive satisfying condition \eqref{Gexp} for some $\alpha\geq0.$ Then there exist the limits
\begin{equation}\label{special value}
\lim_{t\to\infty}\frac{1}{G(e^t)}\sum_{n=0}^{\infty}g(n)^{1+1/t}=\Gamma(\alpha+1),\quad\lim_{t\to\infty}\frac{1}{G(e^t)}\mathrm{Tr}((D_g1)^{1+1/t})=\frac{\Gamma(\alpha+1)}{\log2}.
\end{equation}
\end{proposition}
\begin{proof}
By Lemma \ref{limit lemma}, the first equality in \eqref{special value} is valid trivially.
For every $N\geq1,$ there is $n\geq1$ such that $2^{n}-1\leq N\leq 2^{n+1}-2.$ We have
$$\sum_{k=0}^N\mu(k,D_g1)=\sum_{k=0}^{n-1}2^kg(2^k)+\sum_{k=2^n}^Ng(2^n)=\sum_{k=0}^{n-1}2^kg(2^k)+O(2^ng(2^n)),\quad n\geq1.$$
It follows from \cite[Lemma 4.4]{principal} that
$$\sum_{k=0}^N\mu(k,D_g1)=\frac{G(2^n)}{\log2}+o(G(2^n)),\quad n\to\infty.$$
Clearly, $G(2^n)\sim G(N),$ as $N\to\infty.$ Therefore,
$$\lim_{N\to\infty}\frac{1}{G(N+1)}\sum_{k=0}^N\mu(k,D_g1)=\frac{1}{\log2}.$$
The second limit equality in \eqref{special value} follows from Lemma \ref{limit lemma} as well.
\end{proof}

Lemma \ref{limit lemma} and proposition \ref{special value proposition} imply the following result, which generalises \cite[Theorem 3.3]{CRSS2007}.
\begin{corollary}\label{zeta-lorentz equality} Let $G$ be as in Proposition \ref{special value proposition}, we have $\mathcal{Z}_G =\mathcal{M}_G.$
\end{corollary}

\begin{proof}
We note that the first equality in \eqref{special value} implies that $\mathcal{I}_g\subset\mathcal{Z}_G.$ Since $\mathcal{Z}_G$ is a fully symmetric subspace of $\mathcal{M}_G,$ the assertion follows immediately from \eqref{second definition for lorentz ideal}.
\end{proof}
By Proposition \ref{scalar} and Corollary \ref{special value proposition}, every zeta-function residue with domain $\mathcal{Z}_G=\mathcal{M}_G$ is proportional to a Dixmier trace. By Theorem \ref{biject}, its restriction on $\mathcal{I}_g,$ up to some positive constant,  corresponds to a Banach limit. More precisely, we have the following one-to-one correspondence.
\begin{lemma}\label{blmt0}
Let $g\in\mathcal{G},$  and let $G\in\Omega$ be its primitive function satisfying \eqref{Gexp} with $\alpha\geq0.$ For every $\gamma\in EL(0,\infty),$ the functional $B_{\gamma,g}:=\frac{\log2}{\Gamma(\alpha+1)}\zeta_\gamma\circ D_g$ is a Banach limit and it can be writteen in the following form:
$$
B_{\gamma,g}(x)=\frac{\log2}{\Gamma(\alpha+1)}\gamma\Big(t\mapsto\frac{1}{G(e^t)}\sum_{n=0}^\infty x_n2^{-n/t}\big(2^ng(2^n)\big)^{1+1/t}\Big),\quad x\in\ell_\infty.
$$
Moreover,
$$ \zeta_\gamma(T)=\frac{\Gamma(\alpha+1)}{\log2}B_{\gamma,g}\Big( \Big\{\frac{1}{2^ng(2^n)}\sum_{j=2^n-1}^{2^{n+1}-2}\lambda(j,T)\Big\}_{n\geq0}\Big),\quad T\in\mathcal{I}_g.$$
\end{lemma}
\begin{proof}
Fix $\gamma\in EL(0,\infty).$ Since $\zeta_\gamma$ is a positive singular trace, by Theorem \ref{biject}, $B_{\gamma,g}$ is a positive shift-invariant functional. By Proposition \ref{special value proposition}, $B_{\gamma,g}(1)=1,$ so it is a Banach limit.

Next, for each fixed $s\in(0,\infty),$ we define the continuous function $h_s:[0,1]\rightarrow[0,\infty)$ by setting
$$h_s (a):=a-a^{1+s},\hspace{0.3em} a\in[0,1].$$
Direct calculation yields that
$$\max_{0\leq a\leq1}h_s (a)=h_s \Big(\Big(\frac{1}{s+1}\Big)^{1/s}\Big)=\Big(\frac{1}{s+1}\Big)^{1/s}\frac{s}{s+1}<s.$$
Taking $s=1/t,$ then for $0\leq x\leq1,$ the following equality
\begin{equation}\label{f5}
x_{n}^{1+1/t}=x_{n}+O(1/t)
\end{equation}
holds uniformly in $n$.
Hence,
 \begin{equation*}
 \begin{split}
 \zeta_\gamma (D_gx)
 &=\gamma\Big(t\mapsto\frac{1}{G(e^t)}\sum_{n=0}^\infty 2^n(g(2^n)x_n)^{1+1/t}\Big)\\
 &=\gamma\Big(t\mapsto\frac{1}{G(e^t)}\sum_{n=0}^\infty 2^{-n/t}(2^ng(2^n))^{1+1/t}(x_n+O(1/t))\Big)\\
 &=\gamma\Big(t\mapsto\frac{1}{G(e^t)}\sum_{n=0}^\infty (2^{-n/t}2^ng(2^n))^{1+1/t}x_n\Big),\quad0\leq x\leq1.
 \end{split}
 \end{equation*}
 Since both sides are linear in $x,$ the first assertion follows immediately. The second assertion follows directly from the correspondence established in Theorem \ref{biject}.
\end{proof}

Under an additional assumption $g$ we can write Banach limits from the previous lemma in a convenient form.
\begin{theorem}\label{blmt}
Let $g\in\mathcal{G},$  and let $G\in\Omega$ be its primitive function satisfying \eqref{Gexp} with $\alpha\geq0.$ Suppose in addition that the function $t\mapsto tg(t)$ is eventually monotone. If $B_{\gamma,g}$ is the Banach limit given in Lemma \ref{blmt0} for some $\gamma\in EL(0,\infty),$ then
\begin{equation*}
B_{\gamma,g}(x)=\frac{\log2}{\Gamma(\alpha+1)}\gamma\Big(t\mapsto\frac{1}{G(e^t)}\sum_{n=0}^\infty x_n2^{-n/t}2^ng(2^n)\Big),\quad x\in\ell_\infty.
\end{equation*}
\end{theorem}
\begin{proof}
Set $$f(t):=\sum_{n=0}^\infty2^ng(2^n)\chi_{[n,n+1)}(t),\quad t\geq 0$$
and let $F$ be the primitive function of $f.$
By \cite[Lemma 4.4]{principal}, we have
$$F(t)\sim \frac{G(2^t)}{\log2},\quad t\to\infty.$$
By direct verification, the function $t\mapsto G(2^t)$  belongs to $\mathcal{R}_\alpha.$ Let the function $$\tilde{F}(t):=\int_{0}^\infty e^{-ut}dF(u),\quad t>0$$
denote the Laplace-Stieltjes transform of $F.$ Since $G\in\mathcal{R}_0,$ by \cite[Theorem 1.3.6]{regular}, $$ G(t)=o(t^{\beta}),\quad t\to\infty, \quad \forall\beta>0.$$
Hence, by \cite[Chapter II, Theorem 2.1]{Widder book},
$$\tilde{F}(t)<\infty,\quad \forall t>0.$$
Now using \cite[Chapter IV,~Theorem 8.1]{Tauber}, we obtain
\begin{equation}\label{blmt asymptotic1}
\tilde{F}(\frac{1}{t})\sim \frac{\Gamma(\alpha+1)}{\log2}G(2^t),\quad t\rightarrow\infty.
\end{equation}
Note that
$$\tilde{F}(\frac{\log2}{t})=\int_0^\infty e^{-\frac{u\log2}{t}}dF(u)=\int_0^\infty 2^{-u/t}f(u)du.$$
It follows from \eqref{blmt asymptotic1} that
$$\int_0^\infty 2^{-u/t}f(u)du\sim \frac{\Gamma(\alpha+1)}{\log2}G(e^t),\quad t\rightarrow\infty.$$
Observe that
\begin{equation*}
2^{-1/t}\sum_{n=0}^\infty f(n)2^{-n/t}\leq\int_0^\infty2^{-u/t}f(u)du\leq\sum_{n=0}^\infty f(n)2^{-n/t},\quad t>0.
\end{equation*}
Hence,
\begin{equation}\label{f1}
 \sum_{n=0}^\infty2^{-n/t}f(n)\sim\frac{\Gamma(\alpha+1)G(e^t)}{\log2},\quad t\rightarrow\infty.
\end{equation}
Noting $f(n)=2^ng(2^n),$ $n\geq0,$ we get
\begin{equation*}
\sum_{n=0}^\infty2^{-n/t}f(n)^{1+1/t}=\sum_{n=0}^\infty2^{-n/t}(2^ng(2^n))^{1+1/t}=\mathrm{Tr}((D_g1)^{1+1/t}).
\end{equation*}
By Proposition \ref{special value proposition},
\begin{equation}\label{f2}
\sum_{n=0}^\infty2^{-n/t}f(n)^{1+1/t}\sim \frac{\Gamma(\alpha+1)G(e^t)}{\log2},\quad t\to\infty.
\end{equation}
Now set
$$N_f:=\max\{n\geq0:f(n)\leq1\}.$$
Suppose first that the function $t\mapsto tg(t)$ is nondecreasing. In particular, $n\mapsto f(n)$ is nondecreasing. If $N_f=\infty,$ then it is clear that
$$f(n)\leq1,\quad n\geq0.$$
Thus
$$\big|f(n)^{1+1/t}-f(n)\big|=f(n)-f(n)^{1+1/t},\quad n\geq0,\quad t>0.$$
Combining this with \eqref{f1} and \eqref{f2}, we conclude that
\begin{equation}\label{f3}
\sum_{n=0}^\infty2^{-n/t}\big|f(n)^{1+1/t}-f(n)\big|=\sum_{n=0}^\infty2^{-n/t}f(n)-\sum_{n=0}^\infty2^{-n/t}f(n)^{1+1/t}=o(G(e^t)),\quad t\to\infty.
\end{equation}
If $N_f<\infty,$ then $f(n)>1$ for all $n>N_f,$ and so
$$\big|f(n)^{1+1/t}-f(n)\big|=f(n)^{1+1/t}-f(n),\quad n>N_f,\quad t>0.$$
Write
\begin{equation*}
\sum_{n=0}^\infty2^{-n/t}\big|f(n)^{1+1/t}-f(n)\big|=\big(\sum_{n=0}^{N_f}+\sum_{n=N_f+1}^\infty\big)2^{-n/t}\big|f(n)^{1+1/t}-f(n)\big|.
\end{equation*}
The first summand is bounded above by O(1). The second summand equals
$$\sum_{n=N_f+1}^\infty2^{-n/t}\big(f(n)^{1+1/t}-f(n)\big)=\sum_{n=0}^\infty2^{-n/t}\big(f(n)^{1+1/t}-f(n)\big)+O(1).$$
In view of \eqref{f1} and \eqref{f2}, we also conclude \eqref{f3}.

Next, suppose without loss of generality that $t\mapsto tg(t)$ is decreasing on $[1,\infty)$ and $g(1)=f(0)\leq1.$ In this case, the mapping $n\mapsto f(n)$ is decreasing and
thus $f(n)\leq1,$ $n\geq0.$ Similarly to the case that $t\mapsto tg(t)$ is nondecreasing and $N_f=\infty,$ we conclude that \eqref{f3} holds.

Using \eqref{f3}, for every $x\in\ell_\infty,$ we obtain
$$\sum_{n=0}^\infty x_n2^{-n/t}\big(2^ng(2^n)\big)^{1+1/t}=\sum_{n=0}^\infty x_n2^{-n/t}2^ng(2^n)+o(G(e^t)),\quad t\to\infty.$$
This combined with Lemma \ref{blmt0} completes the proof.
 \end{proof}

The previous theorem enables the following criterion of measurability with respect to all extended zeta-function residues.
\begin{theorem}\label{summation}
Let $g$ and $G$ be as in Theorem \ref{blmt}. An operator $T$ in $\mathcal{I}_g$ is zeta-measurable if and only if there exists the limit
\begin{equation*}
  \lim_{r\to1-}\frac{1}{G(2^{\frac{1}{1-r}})}\sum_{n=0}^\infty a_nr^n,\quad r\in (0,1),\quad a_n:=\sum_{j=2^n-1}^{2^{n+1}-2}\lambda(j,T).
\end{equation*}
\end{theorem}
\begin{proof}
For every $\gamma\in EL(0,\infty),$ let $B_{\gamma,g}:=\frac{\log2}{\Gamma(\alpha+1)}\zeta_\gamma\circ D_g.$
By Theorem \ref{biject} and the Lidskii formula \eqref{Lidf}, we have
\begin{equation*}
\zeta_\gamma(T) =\frac{\Gamma(\alpha+1)}{\log2}B_{\gamma,g}\Big(\Big\{\frac{a_n}{2^ng(2^n)}\Big\}_{n\geq0}\Big).
\end{equation*}
Further, by Theorem \ref{blmt}, we obtain
\begin{equation*}
\zeta_\gamma(T)
 =\gamma\Big(t\mapsto\frac{1}{G(e^t)}\sum_{n=0}^\infty a_n2^{-n/t}\Big).
\end{equation*}
Hence, $T$ is zeta-measurable if and only if there exists the limit
\begin{equation}\label{summation method}
\lim_{t\to\infty}\frac{1}{G(e^t)}\sum_{n=0}^\infty a_n2^{-n/t}.
\end{equation}
Set $r=2^{-1/t}$, we have $t=\frac{\log2}{-\log r}.$ So
$G(e^t)=G(2^{\frac{1}{-\log r}}).$ Thus, \eqref{summation method} is equivalent to the existence of the limit
$$\lim_{r\to1-}\frac{1}{G(2^{\frac{1}{-\log r}})}\sum_{n=0}^\infty a_nr^n.$$
Note that
$$\frac{1}{1-r}-1\leq\frac{1}{-\log r}\leq\frac{1}{1-r}+1,\quad 0<r<1.$$
It follows that
$$\frac12\cdot2^{\frac{1}{1-r}}\leq 2^{\frac{1}{-\log r}}\leq 2\cdot 2^{\frac{1}{1-r}},\quad 0<r<1.$$
Since $G\in\Omega$ satisfy \eqref{cond02}, we conclude that
$$G(2^{\frac{1}{-\log r}})\sim G(2^{\frac{1}{1-r}}),\quad r\rightarrow1-.$$
This completes the proof.
\end{proof}

Below, we list several examples for functions $g\in\mathcal{G}$ satisfying the condition in Theorem \ref{blmt}.
\begin{example}\rm{}\label{ex2}
\begin{enumerate}[(i)]\item The functions
$$\psi_n(t):=\frac{1}{c_n}\chi_{[0,c_n)}(t)+\frac{\log_{[n]}t}{t}\chi_{[c_n,\infty)}(t),\quad n\geq1,$$
where $\log_{[n]}$ is the $n$-th iterations of the logarithmic function and $c_n:=e^{[n]}$ the $n$-th iterations of $\exp.$
\item\label{g_k} The functions $$g_k(t)=\frac{k^k}{e^k}\chi_{[0,e^k)}(t)+\frac{\log^kt}{t}\chi_{[e^k,\infty)}(t),
\quad k\geq0.$$
\item The function
$$g(t)=\frac{1}{e}\chi_{[0,e)}(t)+\frac{1}{t\log t}\chi_{[e,\infty)}(t).$$
\end{enumerate}
\end{example}
\section{The principal ideals $\mathcal{I}_k$}
In this section, we use the relation between asymptotics of zeta function and eigenvalues of an operator to extended measurability result to cover arbitrary operators in more general ideals.

Set
$$g_k(t)=\frac{k^k}{e^k}\chi_{[0,e^k)}(t)+\frac{\log^kt}{t}\chi_{[e^k,\infty)}(t),
\quad k\geq0.$$
and consider ideals $\mathcal{I}_k:=\mathcal{I}_{g_k}.$ For every $k\geq0,$ one easily verifies that the function $G_k(t)=\frac{\log^{k+1}t}{k+1}$ is a primitive of $g_k,$ and satisfies condition \ref{Gexp} with $\alpha=k+1.$ Moreover, we have the following one-to-one correspondence.
\begin{corollary}
For every fixed $\gamma\in EL(0,\infty)$, let $\zeta_\gamma$ be the extended $\zeta$-function residue on $\mathcal{I}_k,$ and let $B_k:=B_{\gamma,g_k}$ be the corresponding Banach limit, then
\begin{equation*}
B_k(x)=\frac{\log^{k+1}2}{k!}\gamma\Big(t\mapsto\frac{1}{t^{k+1}}\sum_{n=0}^\infty x_n2^{-n/t}(n+1)^k\Big),\quad x\in\ell_\infty,
\end{equation*}
and
\begin{equation*}
\zeta_\gamma(T)=\frac{(k+1)!}{\log^{k+1}2}B_k\Big(\Big\{\frac{1}{(n+1)^k}\sum_{j=2^{n}-1}^{2^{n+1}-2}\lambda(j,T)\Big\}_{n=0}^{\infty}\Big),\quad T\in \mathcal{I}_k.
\end{equation*}
\end{corollary}
\begin{proof}
By Theorem \ref{blmt}, we have
\begin{equation}\label{banach limit expression}
\begin{split}
B_{k}(x)&=\frac{\log2}{\Gamma(\alpha+1)}\gamma\Big(t\mapsto\frac{1}{G_k(e^t)}\sum_{n=0}^\infty x_n2^{-n/t}2^ng_k(2^n)\Big)\\
&=\frac{\log2}{k!}\gamma\Big(t\mapsto\frac{1}{t^{k+1}}\sum_{n=0}^\infty x_n2^{-n/t}2^ng_k(2^n)\Big),\quad x\in\ell_\infty.
\end{split}
\end{equation}
If $k=0,$ then $2^ng_k(2^n)=1$ for all $n\geq0.$ The first assertion (and therefore, by Lemma \ref{blmt0}, the second assertion) follows immediately. Let $k\geq1,$ note that
$$2^ng_k(2^n)=\log^k2\cdot n^k,\quad n\geq e^k,$$
and that $(n+1)^k-n^k=O(n^{k-1}).$ For every $x\in\ell_\infty,$ it follows that
$$\log^k2\cdot\sum_{n=0}^\infty x_n2^{-n/t}(n+1)^k=\sum_{n=0}^\infty x_n2^{-n/t}2^ng_k(2^n)+O(\sum_{n=0}^\infty 2^{-n/t}n^{k-1}),\quad t>0.$$
We have
$$\sum_{n=0}^\infty2^{-n/t}n^{k-1}\leq\int_0^{\infty}2^{-s/t} (s+1)^{k-1}\mathrm{~d}s\leq\int_0^12^{-s/t} (s+1)^{k-1}\mathrm{~d}s+2^{k-1}\int_0^{\infty}2^{-s/t} s^{k-1}\mathrm{~d}s .$$
The first summand is bounded above by $O(1),$ the second summand equals $2^{k-1}\frac{\Gamma(k)t^k}{\log^k2}.$
Combining those estimates, we conclude that
$$\log^k2\cdot\sum_{n=0}^\infty x_n2^{-n/t}(n+1)^k=\sum_{n=0}^\infty x_n2^{-n/t}2^ng_k(2^n)+O(t^k),\quad t\to\infty.$$
Therefore,
$$\frac{\log^{k+1}2}{k!}\gamma\Big(t\mapsto\frac{1}{t^{k+1}}\sum_{n=0}^\infty x_n2^{-n/t}(n+1)^k\Big)=\frac{\log2}{k!}\gamma\Big(t\mapsto\frac{1}{t^{k+1}}\sum_{n=0}^\infty x_n2^{-n/t}2^ng_k(2^n)\Big)+\gamma(O(\frac1t)).$$
Comparing this with \eqref{banach limit expression}, the first assertion follows immediately since $\gamma(O(\frac1t))=0.$

Now we prove the second assertion. By Lemma \ref{blmt0}, $$\zeta_\gamma(T)=\frac{(k+1)!}{\log2}B_k\Big(\Big\{\frac{1}{2^ng_k(2^n)}\sum_{j=2^n-1}^{2^{n+1}-2}\lambda(j,T)\Big\}_{n=0}^\infty\Big).$$
It is easy to see that
$$ \frac{1}{2^ng_k(2^n)}\sum_{j=2^n-1}^{2^{n+1}-2}\lambda(j,T)-\frac{1}{\log^k2\cdot(n+1)^k}\sum_{j=2^n-1}^{2^{n+1}-2}\lambda(j,T)=o(1),\quad n\to\infty.$$
Since $B_k$ is Banach limit, the second assertion follows.
\end{proof}

The following lemma is a variant of classical Tauberian theorem (\cite[Chapter I,~Theorem 7.4]{Tauber}) for complex-valued sequences.
\begin{lemma}\label{tauber}
Suppose $(x_n)_{n\geq0}\subset\mathbb{C}$ and $|x_{n}|=O(n^{\alpha-1})$ for some $\alpha\geq0$, if
\begin{equation}\label{tauber1}
\lim_{r\rightarrow1-}(1-r)^{\alpha}\sum_{n=0}^{\infty}x_{n}r^n=c,
\end{equation}
then
$$\lim_{n\to\infty}{n^{-\alpha}}\sum_{j=0}^{n}x_j= \frac{c}{\Gamma(\alpha+1)}.$$
\end{lemma}
\begin{proof}
The assumption $|x_{n}|=O(n^{\alpha-1})$ implies that the series $\sum_{n=0}^{\infty}x_{n}r^n$ converges absolutely for $|r|<1$, and that there exists $\kappa>0$ satisfying $$\Re(x_n),\Im(x_n)\geq-\kappa n^{\alpha-1},\quad n\geq0.$$
Here, $\Re(x_n)$ and $\Im(x_n)$ are real and imaginary parts of $x_n,$ respectively.
By \eqref{tauber1}, we have
$$\lim_{r\rightarrow1-}(1-r)^{\alpha}\sum_{n=0}^{\infty}\Re(x_{n})r^n=\Re(c),$$
Applying the original Tauberian theorem \cite[Theorem 7.4]{Tauber} to the sequence $(\Re(x_n))_{n\geq0},$ we obtain
$$\lim_{n\to\infty}n^{-\alpha}\sum_{j=0}^{n}\Re(x_j)= \frac{\Re(c)}{\Gamma(\alpha+1)}.$$
Similarly, for the imaginary part, we have
$$\lim_{n\to\infty}n^{-\alpha}\sum_{j=0}^{n}\Im(x_j)= \frac{\Im(c)}{\Gamma(\alpha+1)}.$$
Combining those estimates, the assertion follows.
\end{proof}
The following theorem is the main result of this section, which extends \cite[Theorem~4.2]{SUZ2017}.
\begin{theorem}\label{main}
Let $V\in\mathcal{L}_{\infty},~0\leq T\in\mathcal{I}_k$ and $c\in\mathbb{C}.$ The following are equivalent:
\begin{enumerate}[(i)]
  \item\label{main1} All Dixmier traces equal $c$ on $VT$, that is
$$\lim_{n\rightarrow\infty}\frac{1}{ G_{k}(n+1)}\sum_{j=0}^n\lambda(j,VT)=c.$$
  \item\label{main2} The operator $VT$ is $\zeta$-measurable and $\zeta_{\gamma}(VT)=(k+1)!c$ for all $\gamma\in EL(0,\infty).$
  \item\label{main3} There exists the limit
  \begin{equation*}
  \lim_{t\rightarrow\infty}\frac{1}{t^{k+1}}\mathrm{Tr}(VT^{1+1/t})=k!c.
  \end{equation*}
\end{enumerate}
\end{theorem}
\begin{proof}
\eqref{main1}$\Longrightarrow$ \eqref{main2}. For every $\gamma\in EL(0,\infty)$, it follows from Theorem \ref{scalar}, that the functional $\frac{1}{(k+1)!}\zeta_\gamma$ is a Dixmier trace on $\mathcal{I}_k.$ The implication is straightforward.

\eqref{main2}$\Longrightarrow$\eqref{main1}. By Theorem \ref{summation} and its proof, we have
$$ \lim_{r\rightarrow1-}(1-r)^{k+1} \sum_{n=0}^{\infty}a_{n}r^{n}=k!c\log^{k+1}2,$$
where $a_n:=\sum_{j=2^n-1}^{2^{n+1}-2}\lambda(j,VT)$.
Since $|a_n|=O(n^k)$, then using Lemma \ref{tauber}, we obtain

$$\lim_{n\to\infty}\frac{1}{n^{k+1}}\sum_{j=0}^{n}a_{j}=\frac{c\log^{k+1} 2}{k+1},$$
that is
\begin{equation*}
  \lim_{n\to\infty}\frac{1}{G_k(2^n)}\sum_{j=0}^{2^{n+1}-2}\lambda(j,VT)=c.
\end{equation*}
Since $G_k(t)\to\infty$ as $t\to\infty,$ $\sum_{j=2^n-1}^{2^{n+1}-2}\lambda(j,VT)=o(G_k(2^n))$ and
$$\lim_{t\to\infty}\frac{G_k(2t)}{G_k(t)}=1,$$
it follows that
$$\lim_{t\to\infty}\frac{1}{G_k(n+1)}\sum_{j=0}^n\lambda(j,VT)=c.$$

This proves \eqref{main1}.

\eqref{main2}$\Longleftrightarrow$ \eqref{main3}. By \eqref{vzeta} and \eqref{Vzeta}, we have
$$\zeta_{\gamma}(VT)=\zeta_{\gamma,V}(T)=\gamma\left(t\mapsto\frac{k+1}{t^{k+1}}\mathrm{Tr}(VT^{1+1/t}) \right),\quad\forall~\gamma\in EL(0,\infty).$$
Thus, $\zeta_\gamma(VT)=(k+1)!c$ for all $\gamma\in EL(0,\infty)$ if and only if there exists the limit
$$ \lim_{t\rightarrow\infty}\frac{1}{t^{k+1}}\mathrm{Tr}(VT^{1+1/t})=k!c.$$
The proof is complete.
\end{proof}
We'll finish this section with a corollary of the main result.

Let $C:\ell_\infty\rightarrow\ell_\infty$ be the $\mathrm{Ces\grave{a}ro}$ operator defined by
$$(Cx)_n=\frac{1}{n+1}\sum_{i=0}^nx_i,\quad n\geq0.$$
Let $T_0$ be a compact operator such that $\mu(n,T_0)=\frac{1}{n+1},~n\geq0,$ then for every $T\in\mathcal{C}_0(H)$, \cite[Proposition 3.14]{DFWW2004} yields
\begin{equation}\label{tensor}
\mu(T\otimes T_0)\leq C\mu(T)\leq 2\mu(T\otimes T_0).
\end{equation}
For every $k\geq0$, we denote by $\mathcal{L}_{1,\infty}^{\otimes (k+1)}$ the algebraic tensor product
$\underbrace{\mathcal{L}_{1,\infty}\otimes\mathcal{L}_{1,\infty}\otimes\cdots\otimes\mathcal{L}_{1,\infty}}_{\mathrm{k+1~times}}$. Using induction in $k$ and \eqref{tensor}, we have that
$$T\in\mathcal{L}_{1,\infty}^{\otimes (k+1)}\Rightarrow T\in \mathcal{I}_k(H^{\otimes (k+1)}),$$
moreover, $\mathcal{L}_{1,\infty}^{\otimes (k+1)}$ generates $\mathcal{I}_k(H^{\otimes (k+1)})$.
\begin{corollary}\label{terson_product_cor}
Let $ V_0, V_1,\cdots ,V_k\in\mathcal{L}_{\infty}$ and let $0\leq T_0, T_1,\cdots, T_k\in\mathcal{L}_{1,\infty}$. If $V_iT_i,~0\leq i\leq k$ are all Dixmier-measurable, then so is the operator $\bigotimes_{i=0}^kV_iT_i.$ In particular, if $ A_0,A_1,\cdots,A_k\in \mathcal{L}_{1,\infty}$ are Dixmier-measurable, then the operator $\bigotimes_{i=0}^kA_i$ is Dixmier-measurable.
\end{corollary}
\begin{proof}
Since $\mathcal{I}_0=\mathcal{L}_{1,\infty}$, then by \cite[Theorem~4.2]{SUZ2017} or Theorem \ref{main}, there exist $c_0, c_1,\cdots, c_k\in\mathbb{C}$ such that
$$\lim_{t\rightarrow\infty}\frac1t \mathrm{Tr}(V_iT_i^{1+1/t})=c_i,\quad~0\leq i\leq k.$$
For every $\gamma\in EL(0,\infty)$, let $\zeta_\gamma$ be the extended zeta-function residue on $\mathcal{I}_k(H^{\otimes(k+1)}).$

Using \eqref{Vzeta} we obtain
\begin{equation*}
\begin{split}
\zeta_\gamma\Big(\bigotimes_{i=0}^k V_iT_i\Big)&=\zeta_{\gamma,\bigotimes_{i=0}^k V_i}\Big(\bigotimes_{i=0}^kT_i\Big) \\
     & =\gamma\Big( t\mapsto \frac{k+1}{t^{k+1}}\mathrm{Tr}^{\otimes(k+1)}\Big(\Big(\bigotimes_{i=0}^k V_i\Big)\Big(\bigotimes_{i=0}^k T_i\Big)^{1+1/t}\Big)\Big) \\
       & =(k+1)\gamma\Big( t\mapsto \prod_{i=0}^k\frac{1}{t}\mathrm{Tr}\Big(V_iT_i^{1+1/t}\Big)\Big).
\end{split}
\end{equation*}
Since $\gamma$ is multiplicative on the space of convergent (at $+\infty$) functions, we have
\begin{equation*}
      \zeta_\gamma\Big(\bigotimes_{i=0}^k V_iT_i\Big)=(k+1)\prod_{i=0}^kc_i,\quad\forall~\gamma\in EL(0,\infty).
\end{equation*}
By Theorem \ref{main}, we conclude that
$$\lim_{n\rightarrow\infty}\frac{1}{G_k(n+1)}\sum_{j=0}^n\lambda\Big(j,\bigotimes_{i=0}^k V_iT_i\Big)=\frac{1}{k!}\prod_{i=0}^kc_i,$$
the assertion follows.
\end{proof}

We conclude the paper with two applications of Theorem \ref{main} to deduce that certain geometric operators belongs to ideals $\mathcal I_k$ and are Dixmier-measurable.

\begin{example}
	For $n\ge2$, let $-\Delta_n$ be the Dirichlet Laplacian on domain of the form $\Omega_n = \{ x \in \mathbb R^n : |x_1 \cdot \dots \cdot x_n| <1\}.$ It follows from \cite[Theorem 6]{Camus2015} that the eigenvalue counting function $N_{-\Delta_n}$ has the following asymptotic:
	$$N_{-\Delta_n}(t) \sim c_n \cdot t^{n/2} \cdot (\log t)^{n-1}, \ t\to\infty,$$
	where $c_n = n^{n-1} \cdot \left(\Gamma(n/2) \sqrt{\pi^n} 2^{n-1} (n-1)!\right)^{-1}.$
	Thus,
	$$N_{(1-\Delta_n)^{n/2}}(t) \sim C_n \cdot t \cdot (\log t)^{n-1}, \ t\to\infty,$$
	where $C_n = \left(\Gamma(n/2) \sqrt{\pi^n} (n-1)!\right)^{-1}.$
	Since $n_{(1-\Delta_n)^{-n/2}}(t) \sim N_{(-\Delta_n)^{n/2}}(1/t)$ as $t\to 0$, then
	$$n_{(1-\Delta_n)^{-n/2}}(t) \sim C_n \cdot t^{-1} \cdot (\log t^{-1})^{n-1}, \ t\to0.$$
	Note that the asymptotic inverse of the distribution function $n_A$ is the singular value function $\mu(A)$ and that of $t \mapsto t^{-1} \cdot (\log t^{-1})^{n-1}$ is $g_{n-1}(t) = (t+2)^{-1} \cdot (\log (t+2))^{n-1}$. Therefore, the operator $(1-\Delta_n)^{-n/2}$ belongs to ideal $\mathcal I_{n-1}.$
	
	Via Karamata Tauberian theorem \cite[Chapter IV, Theorem 8.1]{Tauber} the asymptotic of $N_{(1-\Delta_n)^{n/2}}$ is equivalent to the following asymptotic of the partition function:
	$${\rm Tr}\left(e^{-t (1-\Delta_n)^{n/2}}\right) \sim C_n \cdot \frac{|\log t|^{n-1}}t, \ t\to0+.$$
	Since the operator zeta-function is the Mellin transform of the partition function, we obtain
	$${\rm Tr}\left( (1-\Delta_n)^{-s n/2} \right) = \frac1{\Gamma(s)}\int_0^\infty {\rm Tr}\left(e^{-t (1-\Delta_n)^{n/2}}\right) t^{s-1} dt \sim C_n \cdot \frac{(n-1)!}{(s-1)^n}, \ \Re(s) > 1.$$
	Now it follows from Theorem \ref{main} that the operator $(1-\Delta_n)^{-n/2}$ is Dixmier-measurable, that is all Dixmier traces coincide on it. Moreover,
	$${\rm Tr}_\omega\left( (1-\Delta_n)^{-n/2} \right) =  \lim_{n\rightarrow\infty}\frac{n+1}{ (\log(k+2))^{n+1}}\sum_{j=0}^k\lambda(j,(1-\Delta_n)^{-n/2})=C_n = \frac1{\Gamma(n/2) \sqrt{\pi^n} (n-1)!}.$$
\end{example}

In Corollary \ref{terson_product_cor} we proved that the tensor product of $k+1$ Dixmier-measurable operators from $\mathcal L_{1,\infty}$ is Dixmier-measurable in $\mathcal I_{k}$. The following example a similar idea to fractal strings.

\begin{example}
	Any bounded open set $\Omega \subset \mathbb R$ can be uniquely written as a countable disjoint union of open intervals of lengths $(l_j)_{j\ge1}$. Following \cite[Section 1]{Lapidus_Frank_book}, we call a (bounded) fractal string both the set $\Omega$ and the associated collection $\mathfrak L = (l_j)_{j\ge1}$. Without loss of generality we can suppose that the sequence $(l_j)_{j\ge1}$ is written in the non-increasing order. For a fractal string one can define the geometric zeta-function by setting
	$$\zeta_{\mathfrak L}(s) := \sum_{j=1}^\infty l_j^s, \ \Re(s)>1.$$
	On the other hand, there is a Dirichlet Laplacian $L$ associated with an open set $\Omega$. Thus, one can define the spectral zeta-function associated with a fractal string by setting
	$$\zeta_{\mathfrak L}^*(s) := \zeta_{L}(-s) = \sum_{j=1}^\infty \lambda(j, L)^{-s}, \ \Re(s)>1.$$
	By \cite[Theoreem 1.19]{Lapidus_Frank_book} these two functions are related as follows:
	$$\zeta_{\mathfrak L}^*(s) = \zeta_{\mathfrak L}(s) \zeta(s) \pi^{-s}, \ \Re(s)>1,$$
	where $\zeta(s)$ is the Riemann zeta-function.
	
	For bounded fractal strings $\mathfrak{L}_1$ and $\mathfrak{L}_2$ the tensor product $\mathfrak{L}_1 \otimes \mathfrak{L}_2 := \{a \cdot b : a\in \mathfrak{L}_1, b\in \mathfrak{L}_2\}$ is again a (bounded) fractal string and by \cite[Lemma 3.3.2]{Lapidus_R_Z_book} the following holds:
	$$\zeta_{\mathfrak{L}_1 \otimes \mathfrak{L}_2}(s) = \zeta_{\mathfrak{L}_1}(s) \cdot \zeta_{\mathfrak{L}_2}(s),$$
	for such $s \in \mathbb C$ that both functions on the right are defined.
	
	Let $\mathfrak L_1$ be the usual Cantor string (of length 1), that is $\mathfrak L_1 = (1/3, 1/9, 1/9, \dots)$ and let $\mathfrak L(1/3, 1/3) := \{3^{-\alpha} : \alpha \in \mathbb N_0\}$ considered as a multiset. Define recursively
	$$\mathfrak L_{n+1} := \mathfrak L_n \otimes \mathfrak L(1/3, 1/3), \ n \ge 1.$$
	By \cite[Example 3.3.7]{Lapidus_R_Z_book} geometric zeta-function $\zeta_{\mathfrak L_{n+1}}$ extends meromorphically to the whole complex plane with poles of order $n+1$ located at $\frac{\log 2}{\log 3} + \frac{2\pi i}{\log 3} \mathbb Z.$  Let $L_{n+1}$ be the Dirichlet Laplacian associated with the fractal string $\mathfrak L_{n+1}$. Then, for $d=\log 2/\log 3$ the zeta-function of an operator $L_{n+1}^{-d}$ extends meromorphically to the whole complex plane with poles of order $n+1$ located at $1 + \frac{2\pi i}{\log 2} \mathbb Z.$ By Theorem \ref{main} the operator $L_{n+1}^{-d}$ is Dixmier-measurable.
\end{example}

\providecommand{\bysame}{\leavevmode\hbox to3em{\hrulefill}\thinspace}
\providecommand{\MR}{\relax\ifhmode\unskip\space\fi MR }
\providecommand{\MRhref}[2]{%
	\href{http://www.ams.org/mathscinet-getitem?mr=#1}{#2}
}
\providecommand{\href}[2]{#2}

\end{document}